\documentclass[12pt]{amsart}

\usepackage{amssymb}

\usepackage{graphicx}

\newtheorem{theorem}{Theorem}
\newtheorem{thmalpha}{Theorem}
\newtheorem{lemma}{Lemma}

\newtheorem{corollary}{Corollary}
\newtheorem*{theorem*}{Theorem}
\theoremstyle{definition}

\theoremstyle{remark}
\newtheorem*{remark*}{Remark}

\newtheorem*{acknowledgment*}{Acknowledgment}

\begin{document}


\title{A Paley-Wiener Type Theorem for Singular Measures on $\mathbb{T}$}
\author{Eric S. Weber}
\address{Department of Mathematics, Iowa State University, 396 Carver Hall, Ames, IA 50011}
\email{esweber@iastate.edu}
\date{\today}
\begin{abstract}
For a fixed singular Borel probability measure $\mu$ on $\mathbb{T}$, we give several characterizations of when an entire function is the Fourier transform of some $f \in L^2(\mu)$.  The first characterization is given in terms of criteria for sampling functions of the form $\hat{f}$ when $f \in L^2(\mu)$.  The second characterization is given in terms of criteria for interpolation of bounded sequences on $\mathbb{N}_{0}$ by $\hat{f}$.  Both characterizations use the construction of Fourier series for $f \in L^2(\mu)$ demonstrated in Herr and Weber \cite{HW17a} via the Kaczmarz algorithm and classical results concerning the Cauchy transform of $\mu$.
\end{abstract}
\maketitle



\section{Introduction}

The classical Paley-Wiener theorem states that an entire function $F$ is the Fourier Transform
for some $f \in L^{2}(-1/2,1/2)$ if and only if $F$ is of exponential type at most $\pi$ and the restriction of $F$ to $\mathbb{R}$ is square-integrable.  An equivalent description of such entire functions is that $F$ satisfy the following two conditions: (i) the values of $F$ on the integer lattice are square-summable and (ii) for every $z \in \mathbb{C}$, $F(z)$ can be recovered from the values of $F$ on the integer lattice via cardinal interpolation, i.e. the Shannon-Whittaker-Kotelnikov Sampling Theorem \cite{BF01a}.
While this latter description of the inhabitants of the Paley-Wiener space is less elegant than the former, we will demonstrate that it is amenable to generalization to singular measures on $\mathbb{T} = (-1/2,1/2)$.\footnote{We assume $\mu(\{1/2\}) = 0$; if necessary, rotate $\mu$ around $\mathbb{T}$ to make it so.}  Indeed, here is the question we shall answer: given a fixed singular Borel probability measure $\mu$ on $\mathbb{T}$, when can an entire function $F$ be written as
\begin{equation} \label{Eq:PW2}
F(z) = \int_{\mathbb{T}} f(x) e^{- 2 \pi i x z} \ d \mu (x)
\end{equation}
for some $f \in L^{2}(\mu)$? Note that we are \emph{a priori} fixing the measure $\mu$.

There are numerous results--in addition to the Paley-Wiener theorem--on when an entire function is the Fourier transform of a function, measure or distribution.  The classical results in this regard include the Plancheral-P\'olya theory \cite{PP37a}, the ``Bochner-Schoenberg-Eberlein conditions'' contained collectively in \cite{Boch34a,Scho34a,Eber55a}, and the Beurling-Malliavin theory \cite{BM62a}.  


We consider the question from the viewpoint of the sampling and interpolating problems for bandlimited functions as presented in \cite{Str00,HuStr01}.  In particular, in \cite{Str00} Strichartz poses a more difficult question than the one we address in the present paper: for a compact set $K$, when is an entire function $F$ the Fourier transform of a (complex) measure supported on $K$?

\section{Main Results}

For our characterization of those functions which admit the representation in Equation (\ref{Eq:PW2}), we require both the Fourier-Stieltjes transform $\widehat{\mu}$ and the Cauchy transform $\mu_{+}$ of $\mu$:
\begin{equation*}
\widehat{\mu}(z) := \int_{\mathbb{T}} e^{- 2 \pi i z x} \ d \mu(x); \qquad \mu_{+}(z) := \int_{\mathbb{T}} \dfrac{ 1 } { 1 - z e^{2 \pi i x} } \ d \mu(x).
\end{equation*}
As noted in \cite{KwMy01}, $\mu_{+}$ is nonvanishing on $\mathbb{D}$; it is in fact the reciprocal of $\mu_{+}$ that we require.  The following theorem uses the Kaczmarz algorithm \cite{Kacz37}, an iterative algorithm for solving systems of linear equations, as well as the main result of \cite{KwMy01}, which says that the Kaczmarz algorithm converges to the solution $x$ for input $\{ \langle x , \varphi_{n} \rangle \}_{n=0}^{\infty}$ when $\{ \varphi_{n} \}$ is a stationary sequence in a Hilbert space with singular spectral measure (see \cite{HW17a} for details; see also \cite{Pol93} for the original proof of the existence of Fourier series for $f \in L^2(\mu)$).

\begin{thmalpha}\label{Th:Fourierseries}
Suppose $\mu$ is a singular Borel probability measure on $\mathbb{T}$, and let $\{ \alpha_{n} \}$ be the sequence of Taylor coefficients of $\dfrac{1}{ \mu_{+}(z)}$.  Define the sequence of functions $g_{n}(x) = \sum_{j=0}^{n} \overline{ \alpha_{n-j} } e^{2 \pi i j x}$.  Then the sequence $\{ g_{n} \}_{n=0}^{\infty} \subset L^2(\mu)$ has the property that for all $f \in L^2(\mu)$,
\begin{equation} \label{Eq:dual}
f = \sum_{n=0}^{\infty} \langle f, g_{n} \rangle e_{n} = \sum_{n=0}^{\infty} \langle f , g_{n} \rangle g_{n}
\end{equation}
with the convergence of both series occuring in the norm.  Moreover, Parseval's identity holds: $\| f \|^{2} = \sum_{n=0}^{\infty} | \langle f, g_{n} \rangle |^{2}$.
\end{thmalpha}
As a consequence of Parseval's identity, the sequence $\{ g_{n} \}_{n=0}^{\infty}$ is a Bessel sequence and hence for any square-summable sequence $\{c_{n}\} \in \ell^{2}(\mathbb{N}_{0})$, the series $\sum_{n=0}^{\infty} c_{n} g_{n}$ also converges in norm.

\subsection{Characterization using Sampling Criteria}
As noted previously, the Paley-Wiener theorem can be reformulated in terms of the Sampling Theorem.  Our first characterization of which entire functions $F = \hat{f}$ for some $f \in L^2(\mu)$ is analogous.

\begin{theorem} \label{Th:sample}
Suppose $\mu$ is a singular Borel probability measure on $\mathbb{T}$, and let $\{ \alpha_{n} \}_{n=0}^{\infty}$ be the Taylor coefficients for $\dfrac{1}{\mu_{+}(z)}$.  An entire function $F$ admits the representation in Equation (\ref{Eq:PW2})
for some $f \in L^2(\mu)$ if and only if the following conditions hold:
\begin{enumerate}
\item[(i)] 
\[ \sum_{n=0}^{\infty} \left| \sum_{j=0}^{n} \alpha_{n-j} F(j) \right|^{2} < \infty; \]
\item[(ii)] for all $z \in \mathbb{C}$,
\begin{equation} \label{Eq:interp} 
F(z) = \sum_{n=0}^{\infty} \left( \sum_{j=0}^{n} \alpha_{n-j} F(j) \right) \left( \sum_{k=0}^{n} \overline{\alpha_{n-k}} \widehat{\mu}(z - k) \right).
\end{equation}
\end{enumerate}
\end{theorem}

\begin{proof}
For $F = \hat{f}$, note that $\sum_{j=0}^{n} \alpha_{n-j} F(j) = \langle f, g_{n} \rangle$, so the necessity of (i) follows by the Parseval identity.  The necessity of (ii) follows by the previous observation and applying the Fourier transform to the second series expansion of $f$ in Equation (\ref{Eq:dual}).

We turn now to the sufficiency.  Combining (ii) with the fact that the sequence $\{ g_{n} \}_{n=0}^{\infty} \subset L^2(\mu)$ is a Bessel sequence, we define the function
\[ f = \sum_{n = 0}^{\infty} \left( \sum_{j=0}^{n} \alpha_{n-j} F(j) \right) g_{n}. \]
As this series converges in $L^2(\mu)$, we obtain
\begin{align*}
\hat{f}(z) &= \sum_{n = 0}^{\infty} \left( \sum_{j=0}^{n} \alpha_{n-j} F(j) \right) \widehat{g_{n}}(z) \\
&= \sum_{n=0}^{\infty} \left( \sum_{j=0}^{n} \alpha_{n-j} F(j) \right) \left( \sum_{k=0}^{n} \overline{\alpha_{n-k}} \widehat{\mu}(z - k) \right) \\
&= F(z)
\end{align*}
by Item (iii).
\end{proof}

\subsection{Characterization using Interpolation Criteria}
We consider now whether the sampling condition in Equation (\ref{Eq:interp}) of Theorem \ref{Th:sample} can be replaced by a different criteria.  Our approach here is to view the characterization from an interpolation viewpoint rather than a sampling viewpoint.   The question then becomes the following: when is $\{ F(n) \}_{n=0}^{\infty}$ the sequence of Fourier moments of some $f \in L^2(\mu)$?  In other words:  given an entire function $F$ does there exist some $f \in L^2(\mu)$ such that for all $n \in \mathbb{N}_{0}$, 
\begin{equation} \label{Eq:interpolate}
F(n) = \int_{\mathbb{T}} f(x) e^{-2 \pi i n x} \ d \mu(x) ?
\end{equation}
Certainly this is a necessary condition for $F = \hat{f}$, and in all (Theorem \ref{Th:type-eps}) but the extremal case (Theorem \ref{Th:typepi}) concerning the support of $\mu$ this is sufficient.

The interpolation problem can be decided using the model subspaces of $H^2(\mathbb{D})$.  For a singular measure $\mu$, there exists a unique inner function $b$ on $\mathbb{D}$ given by the Herglotz Representation \cite{Dur70a}.  This inner function defines a backwards invariant shift invariant space $\mathcal{H}(b) = H^{2} \ominus b H^{2}$ as a consequence of Beurling's theorem \cite{Beu48a}.  Observe that $F \in \mathcal{H}(b)$ if and only if $T_{\overline{b}} F = 0$, where $T_{\varphi}$ is the Toeplitz operator on $H^{2}(\mathbb{D})$ with symbol $\varphi$.  Clark proves \cite{Clark72} that the Normalized Cauchy transform $V_{\mu}$ defined as 
\begin{equation} \label{Eq:NCT}
V_{\mu} : L^2(\mu) \to \mathcal{H}(b) : f \mapsto \dfrac{ \displaystyle{\int_{\mathbb{T}} \dfrac{ f(x) }{ 1 - z e^{-2 \pi i x}} \ d \mu(x) }}{ \mu_{+}(z) } .
\end{equation}
is a unitary operator.   If $F \in \mathcal{H}(b)$, then there exists a unique $f \in L^2(\mu)$ such that $F = V_{\mu}f$.  We denote it by $f = F^{\star}$, and call $F^{\star}$ the $L^2(\mu)$-boundary of $F$ because
\begin{equation} \label{Eq:boundary}
\lim_{r \to 1^{-}} \| f( x ) - F(r e^{2 \pi i x} ) \|_{\mu} = 0.
\end{equation}
This limit was demonstrated by Poltoratski\u\i \ \cite{Pol93} (see also \cite{Aleks89a}).

The Normalized Cauchy Transform can be expressed in terms of the sequence $\{ g_{n} \}$ appearing in Theorem \ref{Th:Fourierseries} \cite{HW17a}:
\begin{equation} \label{Eq:NCTg}
V_{\mu} f (z) = \sum_{n=0}^{\infty} \langle f, g_{n} \rangle z^{n}, \qquad f \in L^2(\mu).
\end{equation}
Therefore, we have the following characterization of the interpolation problem posed in Equation (\ref{Eq:interpolate}).


\begin{lemma} \label{L:interpolate}
Suppose $\mu$ is a singular Borel probability measure on $\mathbb{T}$, $b$ is the inner function on $\mathbb{D}$ associated to $\mu$ via the Herglotz representation, and suppose $\{ a_{n} \}_{n = 0}^{\infty} \subset \mathbb{C}$.  The following conditions are equivalent:
\begin{enumerate}
\item[(i)] there exists a function $f \in L^2(\mu)$ with the property that
\begin{equation} \label{Eq:moments}
a_{n} = \int_{\mathbb{T}} f(x) e^{- 2 \pi i n x} \ d \mu (x);
\end{equation}
\item[(ii)] the following inclusion holds:
\[ G_{a}(z) := \dfrac{ \sum_{n=0}^{\infty} a_{n} z^{n} }{ \mu_{+}(z) } \in \mathcal{H}(b). \]
\end{enumerate}
\end{lemma}

\begin{proof}
Observe that
\begin{equation} \label{Eq:Ga}
G_{a} (z) = \sum_{n=0}^{\infty} \left( \sum_{j=0}^{n} \alpha_{n-j} a_{j} \right) z^{n}.
\end{equation}

Suppose that the moment problem in Equation (\ref{Eq:moments}) has a solution for some $f \in L^2(\mu)$.  Combining Equations (\ref{Eq:Ga}) and (\ref{Eq:NCTg}) demonstrates that $G_{a} = V_{\mu} f$.  Therefore, by (\ref{Eq:NCT}) we obtain that $G_{a} \in \mathcal{H}(b)$.

Conversely, if $G_{a} \in \mathcal{H}(b)$, then reversing the previous argument yields the existence of a function $f \in L^2(\mu)$ such that $G_{a}(z) = \sum_{n=0}^{\infty} \langle f , g_{n} \rangle z^{n}$.  Since we have for every $n$
\[ \sum_{j=0}^{n} \alpha_{n-j} a_{j} = \langle f , g_{n} \rangle = \sum_{j=0}^{n} \alpha_{n-j} \int_{\mathbb{T}} f(x) e^{ - 2 \pi i j x } \ d \mu(x), \]
it now follows that Equation (\ref{Eq:moments}) holds.
\end{proof}

For an entire function $F$ of exponential type, we use $h_{F}$ to denote the Phragm\'en-Lindel\"of indicator function.

\begin{theorem} \label{Th:type-eps}
Suppose $\mu$ is a singular Borel probability measure with support in $[\alpha, \beta] \subset [-1/2,1/2]$ where $\beta - \alpha < 1$. Let $b$ be the inner function associated to $\mu$ via the Herglotz Representation.  The entire function $F$ admits the representation in Equation (\ref{Eq:PW2}) if and only if
\begin{enumerate}
\item[(i)] $F$ is of exponential type;
\item[(ii)] the indicator function of $F$ satisfies $h_{F} ( \dfrac{\pi}{2} ) \leq 2 \pi \beta$ and $h_{F} ( - \dfrac{ \pi }{2} ) \leq - 2 \pi \alpha$;
\item[(iii)] the following inclusion holds:
\[ G_{F}(z) := \dfrac{ \sum_{n = 0}^{\infty} F(n) z^{n} }{ \mu_{+}(z) } \in \mathcal{H}(b) \]
i.e. the function $G_{F}$ is in the kernel of the Toeplitz operator $T_{\overline{b}}$.
\end{enumerate}
\end{theorem}

\begin{proof}
If $F$ admits the representation in Equation (\ref{Eq:PW2}), then $F$ satisfies (i) and (ii) using standard estimates (see, e.g. \cite{PW34}; see also Lemma \ref{L:exponent} below).  Additionally, (iii) follows from Lemma \ref{L:interpolate}.

Conversely, if $F$ satisfies (i), (ii), and (iii), then by Lemma \ref{L:interpolate}, there exists a $f \in L^2(\mu)$ such that
 $\hat{f}(n) = F(n)$ for $n \in \mathbb{N}_{0}$.  Moreover, $\hat{f}$ satisfies (i) and (ii) so we must have $F = \hat{f}$ by Carlson's Theorem (\cite[Theorem 9.2.1]{Boas54a}).
\end{proof}

\begin{corollary}
Suppose $\mu$ is a singular Borel probability measure with support in $[-1/2 + \epsilon, 1/2 - \epsilon]$, and let $b$ be the inner function associated to $\mu$ via the Herglotz Representation.  The entire function $F$ admits the representation in Equation (\ref{Eq:PW2}) if and only if
\begin{enumerate}
\item[(i)] $F$ is of exponential type at most $\pi - \dfrac{1}{2 \epsilon}$;
\item[(ii)] the following inclusion holds:
\[ G_{F}(z) := \dfrac{ \sum_{n = 0}^{\infty} F(n) z^{n} }{ \mu_{+}(z) } \in \mathcal{H}(b) \]
i.e. the function $G_{F}$ is in the kernel of the Toeplitz operator $T_{\overline{b}}$.
\end{enumerate}
\end{corollary}

Recall that we are assuming that $\{ 1/2 \}$ is not an atom for $\mu$.

\begin{lemma}  \label{L:exponent}
If $\mu$ is a Borel measure on $\mathbb{T}$, then for $f \in L^2(\mu)$, $\hat{f}$ satisfies the estimate
\begin{equation} \label{Eq:estimate}
| \hat{f}(z) | \leq \varepsilon (|z|) e^{ \pi | z | }, \qquad \varepsilon( r ) = o(1).
\end{equation}
\end{lemma}

\begin{proof}
For $z = x + iy$ with $y > 0$, we estimate $\| e^{2 \pi i t z } \|_{\mu}^{2}$ as follows:  let $\{ x_{n} \}$ be an increasing sequence with $x_0 = -1/2$, $x_{n} < 1/2$, $x_{n} \to 1/2$ and $\alpha_{n} = \mu( \left( x_{n-1}, x_{n} \right] )$.  We have
\[
\int_{\mathbb{T}} | e^{2 \pi i t z} |^2 d\mu(t) = \sum_{n} \int_{(x_{n-1},x_{n}]} e^{4 \pi t y } d \mu(t) \leq O(1) + \sum_{n} \alpha_{n} e^{ 4 \pi x_{n} y}.
\]
Since $\sum \alpha_{n} < \infty$, we obtain
\[
\dfrac{ \int_{\mathbb{T}} | e^{2 \pi i t z} |^2 d\mu(t) }{ e^{2 \pi y } } \lesssim \sum_{n} \alpha_{n} e^{ 2 \pi (2x_n - 1) y } = o(1)
\]
by Lebesgue's Convergence Theorem.  The same estimate holds for $y < 0$.  Equation (\ref{Eq:estimate}) now follows from the Cauchy-Schwarz inequality.
\end{proof}


\begin{theorem} \label{Th:typepi}
Suppose $\mu$ is a singular Borel probability measure on $\mathbb{T}$, and let $b$ be the inner function associated to $\mu$ by the Herglotz Representation.  The entire function $F$ admits the representation in Equation (\ref{Eq:PW2}) if and only if
\begin{enumerate}
\item[(i)] $| F(z) | \leq \varepsilon(|z|) e^{ \pi | z | }$ with $\varepsilon(r) = o(1)$;
\item[(ii)] the following inclusions hold:
\[ G_{+}(z) := \dfrac{ \sum_{n = 0}^{\infty} F(n) z^{n} }{ \mu_{+}(z) } \in \mathcal{H}(b), \qquad G_{-}(z) := \dfrac{ \sum_{n = 0}^{\infty} \overline{F(-n)} z^{n} }{ \mu_{+}(z) } \in \mathcal{H}(b); \]
\item[(iii)] the $L^2(\mu)$-boundaries of $G_{+}$ and $G_{-}$ satisfy the relationship
\[  \overline{G_{+}^{\star}} = G_{-}^{\star}. \]
\end{enumerate}
\end{theorem}

\begin{proof}
The necessity of Item (i) follows from Lemma \ref{L:exponent}; the necessity of Items (ii) and (iii) are routine.

For the converse, the issue again is when can the sequence $\{ F(n) \}_{n=-\infty}^{\infty}$ be interpolated by a function of the form $\hat{f}$ for some $f \in L^2(\mu)$.
If Item (ii) holds, then by Lemma \ref{L:interpolate}, there exist $f_{+},f_{-} \in L^2(\mu)$ such that for all $n \in \mathbb{N}_{0}$,
\begin{equation}
F(n) = \int_{\mathbb{T}} f_{+}(x) e^{- 2\pi i n x} \ d \mu (x); \qquad \overline{F(-n)} = \int_{\mathbb{T}} f_{-}(x) e^{- 2 \pi i n x} \ d \mu(x).
\end{equation}
Since $V_{\mu} f_{+} = G_{+}$ and $V_{\mu} f_{-} = G_{-}$, we have $f_{+} = G_{+}^{*}$ and $f_{-} = G_{-}^{*}$.  Therefore, if in addition Item (iii) holds, then we have for all $n \in \mathbb{Z}$
\begin{equation}
F(n) = \int_{\mathbb{T}} f_{+}(x) e^{-2 \pi i n x} \ d \mu(x).
\end{equation}
Consequently, if Item (i) holds, $F$ and $\hat{f}_{+}$ are both entire functions satisfying the same estimate by Lemma \ref{L:exponent}.  Therefore, Carlson's Theorem in the form given in \cite[Corollary 9.4.4]{Boas54a} guarantees that $F = \hat{f}$.
\end{proof}

\section{A No-Go Result}

The power and beauty of the Paley-Wiener theorem is that the Paley-Wiener space is the intersection of two simple collections:  entire functions of exponential type at most $\pi$ and $L^2(\mathbb{R})$.  In general, this simplicity cannot be replicated.  There can be no integrability characterization for entire functions $F$ that admit a representation as in Equation (\ref{Eq:PW2}), at least in the sense we make precise presently.  We note that in \cite{Str93a} (see also \cite{Str90a,Str93b}), Strichartz demonstrated that for measures $\mu$ which are ``uniformly $\beta$-dimensional'', the function $\hat{f}$ satisfies the following integrability/Plancherel identity condition:
\begin{equation}
\limsup_{R \to \infty} \dfrac{1}{R^{1-\beta}} \int_{-R}^{R} | \hat{f}(t) |^2 dt \simeq \| f \|_{\mu}^{2}.
\end{equation} 
However, this condition does not characterize the functions $\hat{f}$.

For our purposes here, let us denote $PW(\mu) = \{ \hat{f} | f \in L^2(\mu) \}$, and let us denote by $\mathcal{C}_{\tau}$ the collection of all entire functions of exponential type at most $\tau$ and bounded on $\mathbb{R}$ ($0< \tau \leq \pi$).  By a weight $w$ on $\mathbb{R}$ we mean a nonnegative measurable function; we denote $L^2(w) := \{ f | \int_{\mathbb{R}} | f(x) |^2 w(x) dx < \infty \}$.

\begin{theorem} \label{Th:dB}
Suppose $PW(\mu) = \mathcal{C}_{\tau} \cap L^2(w)$ for some $\tau \in (0,\pi]$ and some weight or measure $w$ on $\mathbb{R}$ with $\| f \|_\mu \simeq \| \hat{f} \|_{w}$.  Then there exists a Riesz basis of the form 
\begin{equation} \label{Eq:RB}
\{ \omega_{n} e^{2 \pi i \lambda_{n} x} \}_{n \in \mathbb{Z}} \subset L^2(\mu)
\end{equation}
for some sequence $\{ \lambda_{n} \} \subset \mathbb{R}$ and $\omega_{n} > 0$. 
\end{theorem}

\begin{proof}
If $PW(\mu) = \mathcal{C}_{\tau} \cap L^2(w)$, then $PW(\mu)$ endowed with the $\| \cdot \|_{w}$ norm is a de Branges space.  Indeed, by \cite[Theorem 23]{dB68a}, we need to verify the conditions identified as (H1), (H2), and (H3).  Clearly $\mathcal{C}_{\tau} \cap L^2(w)$ satisfies (H1) and (H3).  For (H2), note that if $w \in \mathbb{C}$ and $F \in PW(\mu)$, then
\[ | F(w) | \leq \| f \|_{\mu} \| e^{2 \pi i w \cdot} \|_{\mu} \lesssim \| F \|_{w}. \]

Consequently, by \cite[Theorem 22]{dB68a} there exists a sequence $\{ \lambda_{n} \} \subset \mathbb{R}$ such that for all $f \in L^2(\mu)$,
\begin{equation} \label{Eq:sample}
\| \hat{f} \|^{2}_{w} = \sum_{n \in \mathbb{Z}} \dfrac{ | \hat{f}(\lambda_{n}) |^2 }{ K(\lambda_{n}, \lambda_{n}) } \simeq \| f \|^{2}_{\mu}.
\end{equation}
Here $K$ is the reproducing kernel for the space; the sequence of kernels $\{ K(\lambda_{n}, \cdot) \}$ form a complete orthogonal set in the space.  This combined with Equation (\ref{Eq:sample}) proves the claim, where $\omega_{n} = K(\lambda_{n}, \lambda_{n})^{-1/2}$.
\end{proof}

Measures which possess a Riesz basis or frame of exponentials are rare (see e.g. \cite{DHSW11,DL14a}).  The following was demonstrated in \cite{DHW14}:  if $PW(\mu) \subset L^2(w)$ is a closed subspace, then there exists a frame of the form in (\ref{Eq:RB}).


\begin{remark*} We point out that our techniques in the present paper are quite similar to the alternative proof of the Paley-Wiener theorem given in \cite[pg. 106]{Boas54a}.
\end{remark*}

\begin{acknowledgment*}
The results of this paper were inspired while the author attended the workshop ``Hilbert Spaces of Entire Functions and their Applications" at the Institute of Mathematics, Polish Academy of Sciences (IM PAN).  The author thanks the organizers of the workshop for the invitation to participate and IM PAN for their hospitality.
\end{acknowledgment*}


\providecommand{\bysame}{\leavevmode\hbox to3em{\hrulefill}\thinspace}
\providecommand{\MR}{\relax\ifhmode\unskip\space\fi MR }
\providecommand{\MRhref}[2]{%
  \href{http://www.ams.org/mathscinet-getitem?mr=#1}{#2}
}
\providecommand{\href}[2]{#2}


\begin{thebibliography}{DHSW11}

\bibitem[Ale89]{Aleks89a}
A.~B. Aleksandrov, \emph{Inner functions and related spaces of
  pseudocontinuable functions}, Zap. Nauchn. Sem. Leningrad. Otdel. Mat. Inst.
  Steklov. (LOMI) \textbf{170} (1989), no.~Issled. Line\u\i n. Oper. Teorii
  Funktsi\u\i . 17, 7--33, 321. \MR{1039571}

\bibitem[Beu48]{Beu48a}
Arne Beurling, \emph{On two problems concerning linear transformations in
  {H}ilbert space}, Acta Math. \textbf{81} (1948), 17. \MR{0027954}

\bibitem[BF01]{BF01a}
J.~Benedetto and P.J.S.G. Ferriera (eds.), \emph{Modern sampling theory},
  Birkhauser, 2001.

\bibitem[BM62]{BM62a}
A.~Beurling and P.~Malliavin, \emph{On {F}ourier transforms of measures with
  compact support}, Acta Math. \textbf{107} (1962), 291--309. \MR{0147848}

\bibitem[Boa54]{Boas54a}
Ralph~Philip Boas, Jr., \emph{Entire functions}, Academic Press Inc., New York,
  1954. \MR{0068627}

\bibitem[Boc34]{Boch34a}
S.~Bochner, \emph{A theorem on {F}ourier-{S}tieltjes integrals}, Bull. Amer.
  Math. Soc. \textbf{40} (1934), no.~4, 271--276. \MR{1562834}

\bibitem[Cla72]{Clark72}
Douglas~N. Clark, \emph{One dimensional perturbations of restricted shifts}, J.
  Analyse Math. \textbf{25} (1972), 169--191. \MR{0301534 (46 \#692)}

\bibitem[dB68]{dB68a}
Louis de~Branges, \emph{Hilbert spaces of entire functions}, Prentice-Hall
  Inc., Englewood Cliffs, N.J., 1968. \MR{0229011 (37 \#4590)}

\bibitem[DHSW11]{DHSW11}
Dorin~Ervin Dutkay, Deguang Han, Qiyu Sun, and Eric Weber, \emph{On the
  {B}eurling dimension of exponential frames}, Adv. Math. \textbf{226} (2011),
  no.~1, 285--297. \MR{2735759 (2012a:42058)}

\bibitem[DHW14]{DHW14}
Dorin~Ervin Dutkay, Deguang Han, and Eric Weber, \emph{Continuous and discrete
  {F}ourier frames for fractal measures}, Trans. Amer. Math. Soc. \textbf{366}
  (2014), no.~3, 1213--1235. \MR{3145729}

\bibitem[DL14]{DL14a}
Dorin~Ervin Dutkay and Chun-Kit Lai, \emph{Uniformity of measures with
  {F}ourier frames}, Adv. Math. \textbf{252} (2014), 684--707. \MR{3144246}

\bibitem[Dur70]{Dur70a}
Peter~L. Duren, \emph{Theory of {$H^{p}$} spaces}, Pure and Applied
  Mathematics, Vol. 38, Academic Press, New York-London, 1970. \MR{0268655}

\bibitem[Ebe55]{Eber55a}
W.~F. Eberlein, \emph{Characterizations of {F}ourier-{S}tieltjes transforms},
  Duke Math. J. \textbf{22} (1955), 465--468. \MR{0072424 (17,281d)}

\bibitem[HS01]{HuStr01}
Nina~N. Huang and Robert~S. Strichartz, \emph{Sampling theory for functions
  with fractal spectrum}, Experiment. Math. \textbf{10} (2001), no.~4,
  619--638. \MR{1881762 (2003g:94017)}

\bibitem[HW17]{HW17a}
John~E. Herr and Eric~S. Weber, \emph{Fourier series for singular measures},
  Axioms \textbf{6} (2017), no.~2:7, 13 ppg.,
  http://dx.doi.org/10.3390/axioms6020007.

\bibitem[Kac37]{Kacz37}
Stefan Kaczmarz, \emph{Angen\"{a}herte aufl\"{o}sung von systemen linearer
  gleichungen}, Bulletin International de l'Acad\'{e}mie Plonaise des Sciences
  et des Lettres. Classe des Sciences Math\'{e}matiques et Naturelles.
  S\'{e}rie A. Sciences Math\'{e}matiques \textbf{35} (1937), 355--357.

\bibitem[KM01]{KwMy01}
Stanis{\l}aw Kwapie{\'n} and Jan Mycielski, \emph{On the {K}aczmarz algorithm
  of approximation in infinite-dimensional spaces}, Studia Math. \textbf{148}
  (2001), no.~1, 75--86. \MR{1881441 (2003a:60102)}

\bibitem[Pol93]{Pol93}
A.~G. Poltoratski{\u\i}, \emph{Boundary behavior of pseudocontinuable
  functions}, Algebra i Analiz \textbf{5} (1993), no.~2, 189--210, English
  translation in St. Petersburg Math. 5:2 (1994): 389--406. \MR{1223178
  (94k:30090)}

\bibitem[PP37]{PP37a}
M.~Plancherel and G.~P\'olya, \emph{Fonctions enti\`eres et int\'egrales de
  fourier multiples}, Comment. Math. Helv. \textbf{10} (1937), no.~1, 110--163.
  \MR{1509570}

\bibitem[PW87]{PW34}
Raymond E. A.~C. Paley and Norbert Wiener, \emph{Fourier transforms in the
  complex domain}, American Mathematical Society Colloquium Publications,
  vol.~19, American Mathematical Society, Providence, RI, 1987, Reprint of the
  1934 original. \MR{1451142}

\bibitem[Sar94]{Sar94}
Donald Sarason, \emph{Sub-{H}ardy {H}ilbert spaces in the unit disk},
  University of Arkansas Lecture Notes in the Mathematical Sciences, 10, John
  Wiley \& Sons, Inc., New York, 1994, A Wiley-Interscience Publication.
  \MR{1289670 (96k:46039)}

\bibitem[Sch34]{Scho34a}
I.~J. Schoenberg, \emph{A remark on the preceding note by {B}ochner}, Bull.
  Amer. Math. Soc. \textbf{40} (1934), no.~4, 277--278. \MR{1562835}

\bibitem[Str90]{Str90a}
Robert~S. Strichartz, \emph{Self-similar measures and their {F}ourier
  transforms. {I}}, Indiana Univ. Math. J. \textbf{39} (1990), no.~3, 797--817.

\bibitem[Str93a]{Str93a}
\bysame, \emph{Self-similar measures and their {F}ourier transforms. {II}},
  Trans. Amer. Math. Soc. \textbf{336} (1993), no.~1, 335--361.

\bibitem[Str93b]{Str93b}
\bysame, \emph{Self-similar measures and their {F}ourier transforms. {III}},
  Indiana Univ. Math. J. \textbf{42} (1993), no.~2, 367--411.

\bibitem[Str00]{Str00}
\bysame, \emph{Mock {F}ourier series and transforms associated with certain
  {C}antor measures}, J. Anal. Math. \textbf{81} (2000), 209--238. \MR{1785282
  (2001i:42009)}

\end{thebibliography}
\end{document}